\documentclass[12pt,A4paper]{article}
\usepackage{amsmath,amsfonts,amsthm,amssymb, mathtools}
\usepackage{mathrsfs, graphicx,color,latexsym, tikz, calc,cite,enumerate,indentfirst}
\usetikzlibrary{shadows}
\usetikzlibrary{patterns,arrows,decorations.pathreplacing}
\usepackage[colorlinks=true,
linkcolor=blue,citecolor=blue,
urlcolor=blue]{hyperref}
\usepackage[numbers,sort&compress]{natbib}
\newtheorem{theorem}{Theorem}[section]
\newtheorem{lemma}{Lemma}[section]

\newtheorem{conjecture}{Conjecture}[section]

\newtheorem{prop}{Proposition}[section]

\newtheorem{claim}{Claim}
\newtheorem{definition}{Definition}[section]
\theoremstyle{definition}

\allowdisplaybreaks

    \title{\bf Graphs with nonnegative Bakry-\'Emery curvature without Quadrilateral}

\author{Huiqiu Lin,\ Zhe You\setcounter{footnote}{-1}\footnote{\emph{Email address:} huiqiulin@126.com(H. Lin),\ y30231280@mail.ecust.edu.cn (Z. You)}\\[2mm]
	\small School of Mathematics, East China University of Science and Technology, \\
	\small  Shanghai 200237, P. R. China
}

\date{}

\begin{document}
	\maketitle
	\begin{abstract}
The definition of Ricci curvature on graphs in Bakry-\'Emery's sense based on curvature dimension condition was introduced by Lin and Yau [\emph{Math. Res. Lett.}, 2010].
Hua and Lin [\emph{Comm. Anal. Geom.}, 2019] classified unweighted graphs satisfying the curvature dimension condition $CD(0,\infty)$ whose girth are at least five. 
In this paper, we classify all of connected unweighted normalized $C_4$-free graphs satisfying curvature dimension condition $CD(0,\infty)$ for minimum degree at least 2 
and the case with non-normalized Laplacian without degree condition.

		\par\vspace{2mm}
		
		\par\vspace{1mm}
		
		\noindent{\bfseries }
	\end{abstract}

\section{Introduction}\label{section::1}
 Let $G=(V,E)$ be a undirected locally finite simple connected graph and $V,E$ respectively denote the set of vertices and edges.
Let $d_x$ denote the number of edges incident to $x$. We can also assign a weight $m_x$ to each vertex $x$ and a weight $\mu_{xy}$ to each edge $xy$, then $G=(V,E,m,u)$ becomes a weighted graph. The graph $G$ is called \emph{unweighted} if $\mu\equiv 1$ on $E$. For any $x\in V$, we denote $\mu_x:=\sum_{y \sim x} \mu_{x y}$.

We are mostly interested in functions defined on $V$, and denote by $C(V)$ the set of all such functions. For any weighted graph $G$, there is a Laplacian operator, $\Delta: C(V) \rightarrow C(V)$, defined as
$$
\Delta f(x)=\frac{1}{m_x} \sum_{y \sim x} \mu_{x y}(f(y)-f(x)), \quad f \in C(V), x \in V .
$$

The weights $\mu$ and $m$ play important roles in the definition of Laplacian. Given the weight $\mu$ on $E$, typical choices of $m$ are of interest:
\begin{itemize}
\item In case of $m_x=\mu_x$ for all $x \in V$, we call the Laplacian the normalized Laplacian.
\item In case of $m \equiv 1$ on $V$, the Laplacian is called non-normalized (or physical) Laplacian.
\end{itemize}
Moreover, if the graph is unweighted, the corresponding Laplacian is called \textit{unweighted normalized} (i.e. $\mu \equiv 1$ on $E$ and $m \equiv \mu$ on $V$) or \textit{unweighted non-normalized Laplacian} (i.e. $\mu \equiv 1$ on $E$ and $m \equiv 1$ on $V$) respectively. For simplicity, we also call the graph unweighted normalized or unweighted non-normalized graph accordingly.

We denote by $\ell^p(V, m)$ (or simply $\ell_m^p$) the space of $\ell^p$ summable functions on the discrete measure space $(V, m)$ and by $\|\cdot\|_{\ell_m^p}$ the $\ell^p$ norm of a function.

Define the \textit{weighted vertex degree} $\mathrm{D}: V \rightarrow[0, \infty)$ by
$$
\mathrm{D}_x=\frac{1}{m_x} \sum_{y \sim x} \mu_{x y}, \quad x \in V .
$$

It is well known that the Laplacian associated with the graph $G$ is a bounded operator on $\ell_m^2$ if and only if $\sup _{x \in V} \mathrm{D}_x<\infty$.

 As we all know, the Ricci curvature plays a preeminent role in geometric analysis, which has been introduced to general metric measure spaces by Bakry-\'Emery \cite{Bakry}, Ollivier \cite{Ollivier}, Lott-Villani \cite{LV} and Sturm \cite{St1,St2}, etc. Lin, Lu and Yau \cite{LLY} modified Ollivier's definition and introduce a Ricci curvature on graphs. The present paper is devoted to the investigation of the Ricci curvature in  Bakry-\'Emery's sense on graphs that comes from the curvature dimension condition.

The curvature dimension condition $CD(K, N)$, for $K \in \mathbb{R}$ and $N \in(0, \infty]$, on graphs was introduced by \cite{LY}, which serves as the combination of a lower bound $K$ for the Ricci curvature and an upper bound $N$ for the dimension, see section \ref{section2}. Bakry-\'Emery curvature and curvature dimension condition on graphs have been investigated by \cite{Cushing1,Cushing2,Cushing,Liu,Hua1,Hua2,Hua3,Jost,KK,LM,Lin,girth5,LY,LMP1,LMP2,LP,L}.

   Bakry-\'Emery curvature on each vertex of a given graph with a given dimension can be calculated by \cite{Liu}, which provided an useful app for
the computation of discrete curvature at

\centerline{http://www.mas.ncl.ac.uk/graph-curvature/}

Here are some notations. For $x\in V(G)$, let $N(x)$ denote the neighbourhood of $x$ in $G$. The girth of a vertex $x$ in $G$ denoted by $Gir(x)$, is
 defined as the minimal length of cycles passing through $x$ (If there is no cycle passing through $x$ then we define $Gir(x) = \infty$). The girth of $G$ is the minimum length of cycles in $G$.
 Let $\delta(G)$ and $\Delta(G)$ denote the minimum and maximum degree of $G$.
The punctured $2$-ball $\mathring{B}_2(x)=B_2(x)-\{x\}$ denotes the subgraph containing all spherical edges of $S_1(x)$ and all radial edges between $S_1(x)$ and the 2-sphere
 $S_2(x)$ (but not the radial edges of $S_2(x)$, since they have no influence on the curvature function at $x$). Here
 $$S_k(x):=\{y|d(x,y)=k\}\text{, }B_k(x):=\{y|d(x,y)\leq k\}.$$
 $d(x,y)$ is the combinatorial distance between $x$ and $y$.

\section{Preliminaries}\label{section2}
Before introducing Bakry-\'Emery curvature on graphs, some definitions are needed.
\begin{definition}[$\Gamma$ and $\Gamma_2$ operators]
    Let $G=(V, E)$ be a locally finite simple graph. For any two functions $f, g: V \rightarrow \mathbb{R}$, we define
$$
\begin{aligned}
2 \Gamma(f, g) & :=\Delta(f g)-f \Delta g-g \Delta f, \\
2 \Gamma_2(f, g) & :=\Delta \Gamma(f, g)-\Gamma(f, \Delta g)-\Gamma(\Delta f, g) .
\end{aligned}
$$
\end{definition}

For convenience, we write $\Gamma(f):=\Gamma(f, f)$ and $\Gamma_2(f, f):=\Gamma_2(f)$.

Now we can give the definition of Bakry-\'Emery curvature.
\begin{definition}[Bakry-\'Emery curvature]\label{def2.2}
     Let $G=(V, E)$ be a locally finite simple graph. Let $K\in \mathbb{R}$ and $N \in(0, \infty]$. We say that a vertex $x \in V$ satisfies the curvature dimension condition $CD(K,N,x)$, if for any $f\in C(V)$, we have
$$\Gamma_2(f)(x) \geq \frac{1}{N}(\Delta f(x))^2+K \Gamma(f)(x) .$$
We call $K$ a lower Ricci curvature bound of $x$, and $N$ a dimension parameter. The graph $G=(V, E)$ satisfies $C D(K, N)$, if all the vertices satisfy $CD(K, N)$. At a vertex $x \in V$, let $K(G, x ; N)$ be the largest $K$ such that curvature dimension condition holds for all functions $f$ at $x$ for a given $N$. We call $K(G, x ; N)$ the Bakry-\'Emery curvature function of $x$.
\end{definition}

Liu and Peyerimhoff \cite{Cushing} showed that $K(G, x ; N)$ is a non-decreasing function on $N$. Hence, there are more graphs in the situation of dimension $\infty$ than others when considering graphs with non-negative Bakry-\'Emery curvature. That's why we only need to classify graphs in the case of dimension $\infty$ in this paper.
Actually, Bakry-\'Emery curvature comes from Riemannian geometry. Let $(M,\langle\cdot, \cdot\rangle)$ be a Riemannian manifold of dimension $N$ with the Laplacian defined via $\Delta=$ div $\circ$ grad $\leq 0$.
Famous Bochner's formula tell us
$$
\frac{1}{2} \Delta|\operatorname{grad} f|^2(x)=|\operatorname{Hess} f|^2(x)+\langle\operatorname{grad} \Delta f(x), \operatorname{grad} f(x)\rangle+\operatorname{Ric}(\operatorname{grad} f(x)) \text {, }
\forall f \in C^{\infty}(M),$$
where Hess and Ric respectively denote the Hessian and the Ricci tensor. If $\operatorname{Ric}(v) \geq K_x|v|^2$ for all $v \in T_x M$ (tangent space at $x$) and, using the inequality $|\operatorname{Hess} f|^2(x) \geq \frac{1}{N}(\Delta f(x))^2$, we obtain
$$
\frac{1}{2} \Delta|\operatorname{grad} f|^2(x)-\langle\operatorname{grad} \Delta f(x), \operatorname{grad} f(x)\rangle \geq \frac{1}{n}(\Delta f(x))^2+K_x|\operatorname{grad} f(x)|^2 .
$$

The Bakry-\'Emery calculus \cite{Bakry} $\Gamma$ and $\Gamma_2$ for two functions $f, g \in C^{\infty}(M)$ are defined as
$$
\begin{aligned}
2 \Gamma(f, g) & :=\Delta(f g)-f \Delta g-g \Delta f=\langle\operatorname{grad} f, \operatorname{grad} g\rangle, \\
2 \Gamma_2(f, g) & :=\Delta \Gamma(f, g)-\Gamma(f, \Delta g)-\Gamma(g, \Delta f) .
\end{aligned}
$$

Here $\Gamma_2(f)=\frac{1}{2} \Delta|\operatorname{grad} f|^2-\langle\operatorname{grad} \Delta f$, grad $f\rangle$, and we obtain the curvature dimension inequality $\Gamma_2(f)(x) \geq \frac{1}{N}(\Delta f(x))^2+K_x \Gamma(f)(x)$ by using Bochner's formula. Therefore, an $N$-dimensional Riemannian manifold $(M,\langle\cdot, \cdot\rangle)$ with Ricci curvature bounded below by $K_x$ at $x \in M$ satisfies an inequality of the same form given in Definition \ref{def2.2}. It's the motivation to define a notion of Ricci curvature for a metric space, including graphs \cite{LY}, by using this inequality.

Analogous to the case on Riemannian manifolds, Hua and Lin \cite{Hua3} give the Bochner type identity on graphs.

 \begin{prop}[Bochner type identity  \cite{Hua3}]\label{prop2.1}
 For any function $f$ and $x \in V$,
$$
\begin{aligned}
\Gamma_2(f)(x)= & \frac{1}{4}\left|D^2 f\right|^2(x)+\frac{1}{2}(\Delta f(x))^2 \\
& -\frac{1}{4} \sum_{y \sim x} \frac{\mu_{x y}}{m_x}\left(\mathrm{D}_x+\mathrm{D}_y\right)|f(y)-f(x)|^2,
\end{aligned}
$$
where
$$
\left|D^2 f\right|^2(x):=\sum_{\substack{y, z \in V \\ y \sim x, z \sim y}} \frac{\mu_{x y} \mu_{y z}}{m_x m_y}|f(x)-2 f(y)+f(z)|^2 .
$$
 \end{prop}
Lin, Lu and Yau \cite{girth5} classified all the Ricci-flat graphs with girth at least $5$ via Lin-Lu-Yau Ricci curvature. Hua and Lin \cite{Hua3} consider similar problems via Bakry-\'Emery curvature. They classified all the graphs satisfying $CD(0,\infty)$.
\begin{figure}[h]
    \centering
    \includegraphics[width=0.8\textwidth,height=0.2\textwidth]{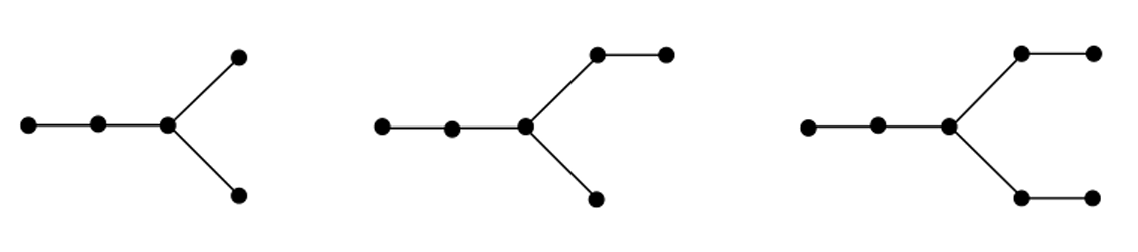}
    \caption{$Star^1_3$, $Star^2_3$ and $Star^3_3$}
    \label{fig1}
\end{figure}
\begin{theorem}[\cite{Hua3}]\label{thm2.1}
     Let $G$ be an unweighted normalized graph with $\inf_{\substack{x \in V}}d_x\geq 2$ and $Gir\left(x_0\right) \geq 5$ for some $x_0 \in V$. Then $G$ satisfies the $CD(0, \infty)$ condition if and only if $G$ is either the infinite line $P_{\mathbb{Z}}$ or the cycle graphs $C_n$ for $n \geq 5$.
\end{theorem}

\begin{theorem}[\cite{Hua3}]\label{thm2.2}
     Let $G$ be an unweighted normalized graph with girth at least 5. Then $G$ satisfies the $CD(0, \infty)$ condition if and only if $G$ is one of the following:\\
(1) The path graphs $P_k(k \geq 1)$, the cycle graphs $C_n(n \geq 5)$, the infinite line $P_{\mathbb{Z}}$, or the infinite half line $P_{\mathbb{N}}$,\\
(2) The star graphs $Star_n(n \geq 3)$ (one vertex is joined to other $n-1$ vertices), or $Star_3^i(1 \leq i \leq 3)$,
where $Star_3^i$ is the 3 -star graph with $i$ edges added, $1 \leq i \leq 3$, shown in figure \ref{fig1}.
\end{theorem}

\begin{theorem}[\cite{Hua3}]\label{thm2.3}
     Let $G$ be an unweighted non-normalized graph with $\inf_{x \in V} \mathrm{~d}_x \geq 2$ and $Gir\left(x_0\right) \geq 5$ for some $x_0 \in V$. Then $G$ satisfies the $CD(0, \infty)$ condition if and only if $G$ is either the infinite line $P_{\mathbb{Z}}$ or the cycle graphs $C_n$ for $n \geq 5$.
\end{theorem}

 \begin{theorem}[\cite{Hua3}]\label{thm2.4}
     Let $G$ be an unweighted non-normalized graph with girth at least five. Then $G$ satisfies the $CD(0, \infty)$ condition if and only if $G$ is one of the following:\\
(1) $P_k(k \geq 1), C_n(n \geq 5), P_{\mathbb{Z}}$, or $P_{\mathbb{N}}$,\\
(2) $Star_3$.
 \end{theorem}

In this paper, we consider $C_4$-free graphs satisfying $CD(0,\infty)$ based on Hua and Lin's results. A graph $G$ is called $F$-free if $F$ is not a subgraph of $G$.

\begin{figure}[h]
    \centering
    \includegraphics[width=0.2\textwidth,height=0.2\textwidth]{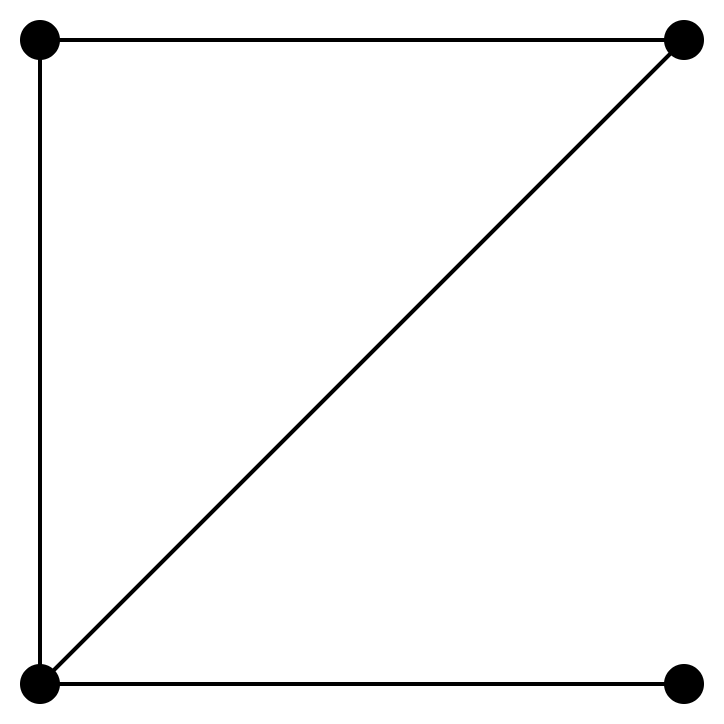}
    \caption{Graph $C'_3$}
    \label{fig2}
\end{figure}
  \begin{theorem}\label{thm2.5}
      Let $G$ be an unweighted non-normalized $C_4$-free graph. Then $G$ satisfies the $CD(0, \infty)$ condition if and only if $G$ is one of the following:\\
      (1) $P_k(k \geq 1), C_n(n \geq 5), P_{\mathbb{Z}}$, or $P_{\mathbb{N}}$,\\
      (2) $\mathrm{Star}_3$,\\
      (3) $C_3$, $C'_3$ (shown in figure \ref{fig2}).
  \end{theorem}
\begin{theorem}\label{thm2.6}
      Let $G$ be an unweighted normalized $C_4$-free graph with $\delta(G)\geq 2$. Then $G$ satisfies the $CD(0, \infty)$ condition if and only if $G$ is one of the following:\\
      (1) $C_n(n \geq 5)$, $P_{\mathbb{Z}}$,\\
      (2) $F_k(1\leq k\leq 7)$, the friendship graph obtained from $k$ triangles by sharing a common vertex.
  \end{theorem}

\section{Proofs of Theorem \ref{thm2.5}\label{section::3}}
    In this section, we classify all the $C_4$-free unweighted graphs with non-normalized Bakry-\'Emery curvature satisfying the curvature dimension condition $CD(0,\infty)$. All the following lemmas here can only be used to consider the non-normalized situation. We assume that we have chosen a specific connected component of $\mathring{B}_2(x)$. We denote the number of vertices of this connected component in $S_1(x)$ and $S_2(x)$ respectively by $r$ and $s$.

\begin{lemma}[\cite{Cushing}]\label{lemma3.1}
     Assume that $\mathring{B}_2(x)$ has more than one connected component, i.e., $d_x>r$, and that $s=0$ and $d_x \geq 4$. Then $K(G, x;\infty)<0$.
\end{lemma}


\begin{lemma}\label{lemma3.2}
   Let $G=(V, E)$ be a locally finite $C_4$-free graph. If $x\in V$ is contained in a triangle and $d_x\geq 3$, then $\mathring{B}_2(x)$ has more than one component.
\end{lemma}
\begin{proof}
Let $x_1,x_2$ be the neighbours of $x$ such that $x,x_1,x_2$ form a triangle. For each $ u\in N(x_1), v\in N(x_2),y\in N(x)\backslash\{x_1,x_2\}$ and $y'\in N(y)\backslash\{x\}$, it is easy to see that $y\nsim u, y\nsim v,y\nsim x_1,y\nsim x_2,y'\nsim x_1,y'\nsim x_2$ since $G$ is $C_4$-free. Therefore, $y$ and $x_1$ lie in different components.
\end{proof}
\begin{theorem}\label{thm3.1}
Let $G$ be a locally finite connected $C_4$-free graph and $G$ is an unweighted non-normalized Laplacian graph. If $G$ satisfies $CD(0,\infty)$ and contains at least one triangle, then $G\cong C_3$ or $C'_3$ shown in \ref{fig2}.
\end{theorem}
\begin{proof}
    Consider $x\in V$ which is contained in a triangle. Let $x_1,x_2$ be the neighbours of $x$ such that $x,x_1,x_2$ form a triangle. If $d_x\geq 3$, we know that $d_x= 3$ if $s=0$  by Lemma \ref{lemma3.1} and \ref{lemma3.2}, and $G\cong C_3'$. If $d_x=2$, then $d_{x_1}=d_{x_2}=2$ and $G\cong C_3$ or one of $d_{x_i}$ equals to $3$ which is same as the case $d_x=3$. Both of $C_3$ and $C_3'$ satisfy $CD(0,\infty)$ whose non-normalized Bakry-\'Emery curvature can be calculated by \cite{Liu}.
\end{proof}
Theorem \ref{thm2.5} is true with Theorem \ref{thm2.4} and \ref{thm3.1}.
\section{Proof of Theorem \ref{thm2.6}\label{section::4}}
In this section, we classify all the $C_4$-free graphs  with $\delta(G)\geq 2$ satisfying the curvature dimension condition $CD(0,\infty)$ with unweighted normalized Laplacian.

\begin{theorem}\label{thm4.1}
Let $G$ be a locally finite connected $C_4$-free graph with $\delta(G)\geq 2$ and $G$ is an unweighted normalized Laplacian graph. If $G$ satisfies $CD(0,\infty)$ and contains at least one triangle, then $G\cong F_k (1\leq k\leq 7)$.
\end{theorem}
   \begin{proof}
    Using Proposition \ref{prop2.1} for unweighted graph with normalized Laplacian, for any function $f\in C(V)$ and $x \in V$, we have
$$
\begin{aligned}
\Gamma_2(f)(x)  =&\frac{1}{4}\left|D^2 f\right|^2(x)+\frac{1}{2}(\Delta f(x))^2 -\frac{1}{4} \sum_{y \sim x} \frac{\mu_{x y}}{m_x}\left(\mathrm{D}_x+\mathrm{D}_y\right)|f(y)-f(x)|^2\\
 =&\frac{1}{4}\sum_{\substack{y, z \in V \\ y \sim x, z \sim y}} \frac{\mu_{x y} \mu_{y z}}{m_x m_y}|f(x)-2 f(y)+f(z)|^2+\frac{1}{2}(\frac{1}{m_x} \sum_{y \sim x} \mu_{x y}(f(y)-f(x)))^2 \\
& -\frac{1}{4} \sum_{y \sim x} \frac{\mu_{x y}}{m_x}\left(\mathrm{D}_x+\mathrm{D}_y\right)|f(y)-f(x)|^2\\
 =&\frac{1}{4}\sum_{\substack{y, z \in V \\ y \sim x, z \sim y}} \frac{1}{d_x\cdot d_y}|f(x)-2 f(y)+f(z)|^2+\frac{1}{2}(\frac{1}{d_x} \sum_{y \sim x} (f(y)-f(x)))^2 \\
& -\frac{1}{4} \sum_{y \sim x} \frac{2}{d_x}|f(y)-f(x)|^2.
\end{aligned}
$$

Since the terms  $\Gamma(f),\Gamma_2(f)$ and $\Delta f$ are all invariant by adding a constant to $f$,
it suffices to check the curvature dimension conditions at $x\in V$ for functions
$f$ satisfying $f(x)=0$.
$$
\begin{aligned}
    CD(0,\infty;x), x\in V
    \Longleftrightarrow &\Gamma_2(f)(x)\geq 0,\forall f\in C(V)\\
    \Longleftrightarrow &\sum_{\substack{y, z \in V \\ y \sim x, z \sim y}} \frac{d_x}{d_y}|f(x)-2 f(y)
    +f(z)|^2\\
    &+2(\sum_{y \sim x} (f(y)-f(x)))^2-2d_x\sum_{y \sim x} |f(y)-f(x)|^2\geq 0\\
    \Longleftrightarrow &\sum_{\substack{y, z \in V \\ y \sim x, z \sim y}} \frac{d_x}{d_y}|2 f(y)
    -f(z)|^2+2(\sum_{y \sim x} f(y))^2-2d_x\sum_{y \sim x} |f(y)|^2\geq 0.\\
    \Longleftrightarrow &
    \sum_{\substack{y, z \in S_1(x) \\ yz\in E(G)}} \frac{d_x}{d_y}|2 f(y)
    -f(z)|^2+\sum_{\substack{y\in S_1(x),z\in S_2(x) \\ z\sim y}} \frac{d_x}{d_y}|2 f(y)
    -f(z)|^2\\
    &+4\sum_{\substack{y\in S_1(x)}} \frac{d_x}{d_y}f(y)^2+2(\sum_{y \sim x} f(y))^2-2d_x\sum_{y \sim x} |f(y)|^2\geq 0.
\end{aligned}
$$
In fact, for $z\in S_2(x)$, we set $f(z)=2f(y)$ if $z\sim y$ and $y\in S_1(x)$. Since $G$ is $C_4$-free, $z\in S_2(x)$ has only one neighbour $y$ in $S_1(x)$, and  the setting above is independent with the values of $f$ on other vertices.
So we only need to consider the following equivalent condition when $G$ is $C_4$-free:
$$
\begin{aligned}
 \sum_{\substack{y, z \in S_1(x) \\ yz\in E(G)}} \frac{d_x}{d_y}|2 f(y)
    -f(z)|^2+4\sum_{\substack{y\in S_1(x)}} \frac{d_x}{d_y}f(y)^2+2(\sum_{y \sim x} f(y))^2-2d_x\sum_{y \sim x} |f(y)|^2\geq 0.
\end{aligned}
$$

Suppose that $x$ is contained in a triangle. To find all the graphs that satisfy conditions, some claims are given as follows.
\begin{claim}\label{claim1}
    For each $y\in N(x)$, $xy$ is contained in a triangle.
\end{claim}
\begin{proof}
    If there exists a vertex $y$ such that $xy$ does not belong to any triangle, we  know that $d_y\geq3$ since $G$ is $C_4$-free.
    Otherwise $Gir(y)\geq 5$, and $G$ is either the infinite line $P_{\mathbb{Z}}$ or the cycle graphs $C_n$ for $n \geq 5$ by Theorem \ref{thm2.1}, which is impossible. Suppose that $x_1$ and $x_2$ are two neighbours of $x$ which form a triangle with $x$.

Now consider a function $f:N(x)\rightarrow \mathbb{R}$ given by
    $$f(z)= \begin{cases}
0 & \text { if } z\in N(x)\backslash\{x_1,x_2,y\} ; \\
a_1 & \text { if } z\in \{x_1,x_2\}; \\
a_2 & \text { if }  z=y.
\end{cases}$$

$$\begin{aligned}
 &\sum_{\substack{p, z \in S_1(x) \\ pz\in E(G)}} \frac{d_x}{d_p}|2 f(p)
    -f(z)|^2+4\sum_{\substack{y\in S_1(x)}} \frac{d_x}{d_p}f(p)^2+2(\sum_{p \sim x} f(p))^2-2d_x\sum_{p \sim x} |f(p)|^2\\
   \leq  & \frac{d_x}{2}a_1^2+\frac{d_x}{2}a_1^2 +4(d_xa_1^2+\frac{d_x}{3}a_2^2)+2(2a_1+a_2)^2-2d_x(2a_1^2+a_2^2)   \\
   =& (a_1^2-\frac{2}{3}a_2^2)d_x+2(a_1+a_2)^2.
\end{aligned}
$$
Here we choose $a_1=-1$ and $a_2=2$ such that
\begin{equation*}
\left\{
\begin{array}{l}
 a_1^2-\frac{2}{3}a_2^2<0,  \\
3(a_1^2-\frac{2}{3}a_2^2)+2(a_1+a_2)^2<0.
\end{array}
\right.
\end{equation*}

    Furthermore, according to distance partitions from $x$, we know that  each vertex in $G$ is contained in at least one triangle.
\end{proof}

\begin{claim}\label{claim2}
    Suppose that $d_x\geq 4$. Then $d_y=2$ for all $y\sim x$.
\end{claim}
\begin{proof}
    Denote the neighbours of $x$ by $x_i$. Without loss of generality, assume that $x_2$ is the neighbour whose degree is at least $4$.
 Consider a  function $f:N(x)\rightarrow \mathbb{R}$ given by
    $$f(z)= \begin{cases}
0 & \text { if } z\in N(x)\backslash\{x_1,x_2,x_3,x_4\} ; \\
c_1 & \text { if } z=x_1; \\
c_2 & \text { if }  z=x_2; \\
c_3 & \text { if }  z=x_3; \\
c_4 & \text { if }  z=x_4.
\end{cases}$$
Now we need to calculate $\Gamma_2(f)$ and choose proper value of $c_i$ to find a contradiction.
$$
\begin{aligned}
 &\sum_{\substack{y, z \in S_1(x) \\ yz\in E(G)}} \frac{d_x}{d_y}|2 f(y)
    -f(z)|^2+4\sum_{\substack{y\in S_1(x)}} \frac{d_x}{d_y}f(y)^2+2(\sum_{y \sim x} f(y))^2-2d_x\sum_{y \sim x} |f(y)|^2\\
   \leq  &\frac{d_x}{2}(2c_1-c_2)^2+\frac{d_x}{4}(2c_2-c_1)^2+\frac{d_x}{2}(2c_3-c_4)^2+\frac{d_x}{2}(2c_4-c_3)^2\\
   &+2(c_1+c_2+c_3+c_4)^2-c^2_2d_x.
\end{aligned}
$$
Due to  homogeneity of $c_i$ above, without loss of generality, let $c_4=1$. Assume that
$$\frac{1}{2}(2c_1-c_2)^2+\frac{1}{4}(2c_2-c_1)^2+\frac{1}{2}(2c_3-c_4)^2+\frac{1}{2}(2c_4-c_3)^2-c^2_2<0.$$
Consider the right side of the inequality,
$$
\begin{aligned}
    & \frac{d_x}{2}(2c_1-c_2)^2+\frac{d_x}{4}(2c_2-c_1)^2+\frac{d_x}{2}(2c_3-1)^2+\frac{d_x}{2}(2-c_3)^2+2(c_1+c_2+c_3+1)^2-c^2_2d_x\\
   \leq & 2(2c_1-c_2)^2+2(2c_2-c_1)^2+2(2c_3-1)^2+2(2-c_3)^2+2(c_1+c_2+c_3+1)^2-4c^2_2\\
   =& (c_1+2)^2+(c_2+2)^2+6(c_3-1)^2+2(c_2-2c_1)^2+(c_2+c_3)^2+2(c_1+c_3)^2-2.
\end{aligned}$$

Choose $c_i$ such that the inequality above is less than $0$ and the assumption is right. Here we choose $c_1=-1$, $c_2=-2$, $c_3=c_4=1$.
Then
$$\begin{aligned}
 &\sum_{\substack{y, z \in S_1(x) \\ yz\in E(G)}} \frac{d_x}{d_y}|2 f(y)
    -f(z)|^2+4\sum_{\substack{y\in S_1(x)}} \frac{d_x}{d_y}f(y)^2+2(\sum_{y \sim x} f(y))^2-2d_x\sum_{y \sim x} |f(y)|^2\\
   \leq  &  -\frac{3}{4}d_x+2\leq-1.
   \end{aligned}$$
   It is a contradiction with $CD(0,\infty)$. Since $G$ is $C_4$-free, $d_x$ is an even number by Claim \ref{claim1}. Thus $d_y=2$ for all $y\sim x$.
\end{proof}
Now we get a structure of friendship graphs, and what we only need to prove next is $ \Delta(G)\leq 14$. Here we suppose $d_x=2$. $y,z$ is two neighbours of $x$ that $d_z=2$ and $d_y=\Delta(G)\geq 16$.
Consider a  function $f:N(x)\rightarrow \mathbb{R}$ given by
    $$f(p)= \begin{cases}
b_1 & \text { if } p=y; \\
b_2 & \text { if }  p=z.
\end{cases}$$
$$\begin{aligned}
 &\sum_{\substack{y, z \in S_1(x) \\ yz\in E(G)}} \frac{d_x}{d_y}|2 f(y)
    -f(z)|^2+4\sum_{\substack{y\in S_1(x)}} \frac{d_x}{d_y}f(y)^2+2(\sum_{y \sim x} f(y))^2-2d_x\sum_{y \sim x} |f(y)|^2\\
   \leq  &  \frac{49}{8}b^2_2-\frac{1}{2}b_1b_2=(\frac{49}{8}b_2-\frac{1}{2}b_1)\cdot b_2.
   \end{aligned}$$
Choose $b_2=1$ and $b_1=13$, then $$\frac{49}{8}b^2_2-\frac{1}{2}b_1b_2<0.$$
A contradiction with $CD(0,\infty)$. Thus $\Delta(G)\leq 15$. By Claim \ref{claim1}, $\Delta(G)\neq 15$, so $\Delta(G)\leq 14$. And we can calculate normalized Bakry-\'Emery curvature of $F_k$ $(1\leq k\leq 7)$ by \cite{Liu}, which tells us that all of them satisfy $CD(0,\infty)$.
\end{proof}

Theorem \ref{thm2.6} can be got immediately by Theorem \ref{thm2.2} and \ref{thm4.1}.

\section{Concluding remark}\label{section::5}
We need  a condition $\delta(G)\geq2$ to complete our proof on unweighted normalized Laplacian. Actually, it's such a weak condition that we can remove it and complete the whole classification, which is a easy but cumbersome work based on our proof. There is a similar question on triangle-free graph. However, it seems to be impossible to classify all of triangle-free graph satisfying $CD(0,\infty)$ since we can find a lot of examples with unweighted normalized and non-normalized Laplacian by \cite{Liu}. In \cite{KK}, they calculate non-normalized Bakry-\'Emery curvature of hypercube, which is also a regular graph. It is clear that a regular graph satisfies $CD(0,\infty)$ with non-normalized Laplacian if and only if it satisfies $CD(0,\infty)$ with normalized Laplacian. Taking a step back, we want to know what properties and structures triangle-free graphs with non-negative Bakry-\'Emery curvature have. Using the curvature calculator, we find that many of these graphs consist of 4-cycle. Therefore, we pose the following conjecture:
\begin{conjecture}
   Triangle-free graphs satisfying $CD(0,\infty)$ with unweighted normalized and non-normalized Laplacian contain no induced $k$-cycle $(k\geq 5)$, except $C_n(n\geq5)$ and several graphs.
\end{conjecture}


\end{document}